\documentclass[12pt,a4paper]{amsart}

\usepackage[left=1in,top=1.25in,centering]{geometry}

\usepackage{amssymb,amscd,amsxtra,calc,cite}

\usepackage{microtype}

\usepackage[all]{xy}

\usepackage[linktocpage]{hyperref}
\hypersetup{
    pdftitle={},            
    pdfauthor={Dinh-Hu-Zhang},   
    colorlinks=true,        
    linkcolor=red,          
    citecolor=blue,         
    filecolor=magenta,      
    urlcolor=cyan           
}

\theoremstyle{plain}
\newtheorem{thm}{Theorem}[section]
\newtheorem{claim}[thm]{Claim}

\newtheorem{corollary}[thm]{Corollary}

\newtheorem{lemma}[thm]{Lemma}

\newtheorem{notation}[thm]{Notation}
\newtheorem{proposition}[thm]{Proposition}

\newtheorem{problem}[thm]{Problem}
\newtheorem{theorem}[thm]{Theorem}

\theoremstyle{definition}
\newtheorem{definition}[thm]{Definition}

\newtheorem{remark}[thm]{Remark}

\theoremstyle{remark}

\newcommand{\bC}{\mathbb{C}}
\newcommand{\bN}{\mathbb{N}}

\newcommand{\bQ}{\mathbb{Q}}
\newcommand{\bR}{\mathbb{R}}
\newcommand{\bZ}{\mathbb{Z}}
\newcommand{\CE}{\mathcal{E}}
\newcommand{\CK}{\mathcal{K}}

\newcommand{\oo}{\operatorname{o}}
\newcommand{\OO}{\operatorname{O}}
\newcommand{\alb}{\operatorname{alb}}

\newcommand{\Aut}{\operatorname{Aut}}

\newcommand{\GL}{\operatorname{GL}}

\newcommand{\id}{\operatorname{id}}
\newcommand{\Imm}{\operatorname{Im}}
\newcommand{\Ker}{\operatorname{Ker}}

\newcommand{\Nef}{\operatorname{Nef}}
\newcommand{\NS}{\operatorname{NS}}

\newcommand{\Alb}{\operatorname{Alb}}

\newcommand{\isom}{\simeq}

\begin{document}

\title[Compact K\"ahler manifolds admitting large solvable groups]
{Compact K\"ahler manifolds admitting large solvable groups of automorphisms}

\author{Tien-Cuong Dinh}
\author{Fei Hu}
\author{De-Qi Zhang}

\address
{\textsc{Department of Mathematics} \endgraf
\textsc{National University of Singapore, 10 Lower Kent Ridge Road, Singapore 119076}}

\email{\href{mailto:matdtc@nus.edu.sg}{matdtc@nus.edu.sg}}
\email{\href{mailto:hf@u.nus.edu}{hf@u.nus.edu}}
\email{\href{mailto:matzdq@nus.edu.sg}{matzdq@nus.edu.sg}}

\begin{abstract}
Let $G $ be a group of automorphisms of a compact K\"ahler manifold $X$ of dimension $n$ and $N(G)$ the subset of null-entropy elements.
Suppose $G$ admits no non-abelian free subgroup.
Improving the known Tits alternative, we obtain that, up to replace $G$ by a finite-index subgroup,
either $G/N(G)$ is a free abelian group of rank $\le n-2$,
or $G/N(G)$ is a free abelian group of rank $n-1$ and $X$ is a complex torus, or $G$ is a free abelian group of rank $n-1$.
If the last case occurs, $X$ is $G$-equivariant birational to the quotient of an abelian variety provided that $X$ is a projective manifold of dimension $n\geq 3$ and is not rationally connected.
We also prove and use a generalization of a theorem by Fujiki and Lieberman on the structure of $\Aut(X)$.
\end{abstract}


\maketitle

\section{Introduction}

We work over the field $\bC$ of complex numbers. Let $X$ be a compact K\"ahler manifold of dimension $n$ and $\Aut(X)$ the group of all (holomorphic) automorphisms of $X$.
Fujiki and Lieberman proved that $\Aut(X)$ is a complex Lie group of finite dimension but it may have an infinite number of connected components.
The identity connected component $\Aut_0(X)$ of $\Aut(X)$ is the set of automorphisms obtained by integrating holomorphic vector fields on $X$.
It is a normal subgroup of $\Aut(X)$ and its action on the Hodge cohomology groups $H^{p,q}(X, \bC)$ is trivial. See \cite{Fujiki78, Lieberman78} for details.

Elements in $\Aut_0(X)$ have zero (topological) entropy. Indeed, for any automorphism $g\in\Aut(X)$, the {\it entropy} of $g$, defined in the theory of dynamical systems,
turns out to be equal to the logarithm of the spectral radius of the pull-back operator $g^*$ acting on $\oplus_{0\leq p\leq n} H^{p,p}(X,\bR)$.
This is a consequence of theorems due to Gromov and Yomdin.
Topological entropy is always a non-negative number and also equals the logarithm of the spectral radius of $g^*$ acting on the whole cohomology group $\oplus_{0\leq p,q\leq n} H^{p,q}(X, \bC)$.
This number is strictly positive if and only if the spectral radius of $g^*$ acting on $H^{p,p}(X,\bR)$ is strictly larger than $1$ for all or for some $p$ with $1\leq p\leq n-1$.
See \cite{Dinh12, Gromov03, Yomdin87} for details.

The group $\Aut(X)$ satisfies the following Tits alternative type result which was proved in \cite[Theorem 1.5]{CWZ14} (generalizing \cite[Theorem 1.1]{Zhang-Invent}). See \cite{Dinh12} for a survey.
Recall that a group $H$ is {\it virtually solvable} (resp. {\it free abelian}, {\it unipotent...}), if a finite-index subgroup of $H$ is solvable (resp. free abelian, unipotent...).

\begin{theorem}\label{Z-TitsTh}
Let $X$ be a compact K\"ahler manifold of dimension $n\ge 2$ and $G \le \Aut(X)$ a group of automorphisms. Then one of the following two alternative assertions holds:
\begin{itemize}
\item[(1)] $G$ contains a subgroup isomorphic to the non-abelian free group $\bZ * \bZ$, and hence $G$ contains subgroups isomorphic to non-abelian free groups of all countable ranks.
\item[(2)] $G$ is virtually solvable.
\end{itemize}
In the second case, or more generally, when the representation $G|_{H^2(X,\bC)}$ is virtually solvable, $G$ contains a finite-index subgroup $G_1$ such that the {\bf null-entropy subset}
$$N(G_1) := \{g \in G_1 \, | \, g \,\, \text{\rm is of null entropy}\}$$
is a normal subgroup of $G_1$ and the quotient $G_1/N(G_1)$ is a free abelian group of rank $r \le n-1$.
\end{theorem}

When the action of $G$ on $H^2(X,\bC)$ is virtually solvable, the integer $r$ in the theorem is called the {\it dynamical rank} of $G$ and we denote it as $r = r(G)$.
It does not depend on the choice of $G_1$. When $G|_{H^2(X,\bC)}$ is abelian,
Theorem \ref{Z-TitsTh} can be deduced from \cite[Theorem I]{DS04} and the classical Tits alternative for linear algebraic groups.
In \cite[Remarque 4.9]{DS04}, the authors mentioned the interest of studying these $X$ admitting a commutative $G$ of positive entropy and maximal dynamical rank $n-1$.

In this paper, we study virtually solvable but not necessarily commutative groups $G$ of maximal dynamical rank $r(G)=n-1$.
From Theorem \ref{ThA'} below, it turns out that $N(G)$ is then virtually contained in $\Aut_0(X)$, hence $G /(G \cap \Aut_0(X))$ is virtually a free abelian group of rank $n-1$.
In particular, \cite[Question 2.17]{Zhang-Invent} is affirmatively answered by Theorem \ref{ThA'} (or Theorem \ref{ThA}) below.

\begin{theorem}\label{ThA'}
Let $X$ be a compact K\"ahler manifold of dimension $n \ge 2$ and $G \le \Aut(X)$ a group of automorphisms
such that the null-entropy subset $N(G)$ is a subgroup of (and hence normal in) $G$ and $G/N(G) \cong \bZ^{\oplus n-1}$.
Then either $X$ is a complex torus, or $N(G)$ is a finite group and hence $G$ is virtually a free abelian group of rank $n-1$.
\end{theorem}

Note that $N(G)$ is normal in $G$ because if $g\in G$ has zero entropy, by Gromov and Yomdin, $hgh^{-1}$ also has zero entropy for every $h\in \Aut(X)$.

\begin{remark}\label{rThA'}
We can deduce from Theorem \ref{Z-TitsTh} that the representation $G|_{H^2(X, \bC)}$ is virtually solvable if and only if, for some finite-index subgroup $G_1$ of $G$,
$N(G_1)$ is a normal subgroup of $G_1$ and $G_1/N(G_1) \cong \bZ^{\oplus r}$ for some $r\le n-1$ (cf. \cite[Theorem 2.3]{CWZ14}). So the condition of Theorem \ref{ThA'} is,
up to replace $G$ by a finite-index subgroup, equivalent to that $G|_{H^2(X, \bC)}$ is virtually solvable and $G$ has maximal dynamical rank $n-1$.
Proposition \ref{PropA} below gives a more precise description of the present situation.
\end{remark}

Recall that by Theorem \ref{Z-TitsTh}, the first hypothesis in the following result is satisfied when $G$ admits no non-abelian free subgroup.
Note also that Theorem \ref{ThA'} can be applied to the following group $G_0$.

\begin{proposition} \label{PropA}
Let $X$ be a compact K\"ahler manifold of dimension $n \ge 2$ and $G \le \Aut(X)$ a group of automorphisms such that $G|_{H^2(X,\bC)}$ is virtually solvable and of maximal dynamical rank $n-1$.
Then there is a normal subgroup $G_0$ of $G$ such that
\begin{itemize}
\item[(1)] $G/G_0$ is isomorphic to a subgroup of the symmetric group $S_n$. In particular, it contains at most $n!$ elements.
\item[(2)] The null-entropy subset $N(G_0)$ is a normal subgroup of $G_0$ and $G_0/N(G_0)\cong\bZ^{\oplus n-1}$.
\end{itemize}
\end{proposition}

\begin{remark}\label{rPropA}
Let $X, G$ be as in Proposition \ref{PropA}, and let $H$ be the normalizer of $G$ in $\Aut(X)$. Then either $|H : G| < \infty$, or $X$ is a complex torus.
And in the latter case $|H/(H \cap \Aut_0(X)) : G/(G \cap \Aut_0(X))| < \infty$.
\end{remark}

Theorem \ref{ThA'} (or Theorem \ref{ThA} (2)) enables us to apply \cite[Theorem 1.1]{Zhang-TAMS15} and immediately get the following result.

\begin{corollary}\label{CorB}
Assume the condition for $X$ and $G$ in Theorem \ref{ThA'} (or \ref{ThA}). Assume further that $n \ge 3$, $X$ is a projective manifold and $X$ is not rationally connected.
Then, replacing $G$ by a finite-index subgroup, we have the following:
\begin{itemize}
\item[(1)] There is a birational map $X \dashrightarrow Y$ such that the induced action of $G$ on $Y$ is biregular and $Y = T/F$,
where $T$ is an abelian variety and $F$ is a finite group whose action on $T$ is free outside a finite subset of $T$.
\item[(2)] There is a faithful action of $G$ on $T$ such that the quotient map $T \to T/F = Y$ is $G$-equivariant. Every $G$-periodic irreducible proper subvariety of $Y$ or $T$ is a point.
\end{itemize}
\end{corollary}

When $X$ is rationally connected,
we can also state some necessary and sufficient hypothesis for $X$ to be $G$-equivariant birational to a quotient of an abelian variety, see \cite{Zhang-TAMS15} for more details.

Generalizing Theorem \ref{ThA} and Corollary \ref{CorB}, we propose the following. See \cite[Problem 1.5]{Dinh12} and the remarks afterwards for a related problem.
See Oguiso and Truong \cite[Theorem 1.4]{OT13} for the first example of rational threefold or Calabi--Yau threefold admitting a primitive automorphism of positive entropy.

\begin{problem}
\begin{itemize}
\item[(1)] Classify compact K\"ahler manifolds of dimension $n \ge 3$ admitting a solvable group $G$ of automorphisms of dynamical rank $r$ for some $2\leq r\leq n-1$,
such that the pair $(X, G)$ is primitive, i.e., there is no non-trivial $G$-equivariant meromorphic fibration even after replacing $G$ by a finite-index subgroup.
\item[(2)] One may strengthen the condition on $G$ to that every non-trivial element of $G$ is of positive entropy. So $G$ is isomorphic to $\bZ^{\oplus r}$.
\item[(3)] One may also weaken the condition on $G$ to that the representation $G |_{H^{1,1}(X, \bC)}$ is an abelian or solvable group.
\end{itemize}
\end{problem}

One may ask similar problems for the case of manifolds defined over fields of positive characteristic but new tools need to be developed.
We refer the reader to a pioneering work in this direction by Esnault and Srinivas in \cite{ES13}.

The present paper is organized as follows. In Section \ref{section_Fu_Li}, we give a generalization of a theorem by Fujiki and Lieberman that we will need in the proof of the main theorem.
Theorem \ref{th_Fu_Li} and Proposition \ref{LF} should be of independent interest.
In Section \ref{section_privileged}, we study a notion of privileged vector space which allows us to construct common eigenvectors with maximal eigenvalues.
The main theorem and its proof, as well as the proofs of Proposition \ref{PropA} and Remark \ref{rPropA} are given in the last section.

\par \vskip 1pc \noindent
{\bf Idea of the proof of the main theorems \ref{ThA'} and \ref{ThA}.} Assume $G/N(G)$ is free abelian of maximal rank $n-1$.
The key step is to show that the null-entropy subset $N(G)$ is virtually contained in the identity connected component $\Aut_0(X)$ of $\Aut(X)$,
i.e., $N(G)$ is a finite extension of $N(G) \cap \Aut_0(X)$.
Indeed, once this assertion is proved, either $\Aut_0(X)$ is trivial and hence $G$ is virtually free abelian of rank $n-1$, or $\Aut_0(X)$ is non-trivial.
In the latter case, since the maximality of the dynamical rank of $G$ implies that every $G$-equivariant fibration is trivial, we can show that

\vskip 1em
-- either $X$ is almost homogeneous under a linear algebraic group and hence $G$ is of null entropy, which contradicts the fact that the dynamical rank of $G$ is $n-1 \ge 1$ (cf. \cite{FZ13}).

\vskip 1em
-- or the albanese map $\alb_X : X \to \Alb(X)$ is bimeromorphic (and indeed is biholomorphic), so $X$ is a complex torus.

\vskip 1em
The desired assertion is equivalent to that the representation of $N(G)$ on $H^2(X, \bC)$ is a finite group, in view of the results of Fujiki and Lieberman \cite{Fujiki78,Lieberman78}.
However, their theorems cannot be directly applied in our setting. We will give in Theorem \ref{th_Fu_Li} or Proposition \ref{LF} some generalization of their results.
Hence, for our purpose, we only need to prove that after replacing $G$ by a finite-index subgroup, the null-entropy subset $N(G)$ fixes a nef and big class $A$ in $H^{1,1}(X,\bR)$. Such a class $A$ will be constructed as follows (see \hyperref[pfThA]{Proof of Theorem \ref{ThA} (1)}).

Note the representation $N(G) |_{H^2(X, \bC)}$ is virtually a unipotent group (say a unipotent group for simplicity), and hence has a non-trivial centre $Z|_{H^2(X,\bC)}$ for some $Z\unlhd G$.
We then introduce a notion of privileged subspaces of $H^2(X,\bR)$, defined over $\bQ$, in which we can control classes in $H^{2,0}(X,\bC)\oplus H^{0,2}(X,\bC)$ by classes in $H^{1,1}(X,\bR)$, see Definition \ref{def_priv}.
The key Proposition \ref{priv} shows the existence of a $G$-invariant privileged subspace $F$ on which the action of $Z$ is trivial.
By considering the upper central series of $N(G) |_{H^2(X, \bC)}$ and doing induction on its length, we can shrink $F$ such that the action of $N(G)$ on such $F$ is also trivial.

It follows from the above discussion that $G|_F$ is abelian. Following the spirit of \cite{DS04}, we can find $s+1$ nonzero common nef classes $L_1,\dots,L_{s+1}\in F$ of $G|_F$, such that $L_1\cdots L_{s+1}\ne 0$ in $H^{s+1,s+1}(X,\bR)$ and the group homomorphism below
$$\psi:G|_F\to (\bR^{\oplus s},+), \quad g^*|_F\mapsto (\chi_1(g),\dots,\chi_{s}(g))$$
satisfies that $\Ker\psi = N(G|_F)$ and $\Imm \psi$ is a discrete subgroup of $\bR^{\oplus s}$ generating $\bR^{\oplus s}$, where each $\chi_i:G\to (\bR,+)$ is the corresponding group character of $L_i$.
In other words, we have
$$G|_F/N(G|_F)\isom \bZ^{\oplus s}.$$

We next claim that the above $s=n-1$. Otherwise, we can extend those $L_i$ to a full quasi-nef sequence as in \cite[\S 2.7]{Zhang-Invent}.
Let $\widetilde N$ denote the inverse image of $N(G|_F)$ under the natural restriction $\pi:G\to G|_F$.
It is easy to see that $\widetilde N/N(G)$ is a free abelian group of rank $n-1-s$.
While by analysing the image of $N(G|_F)$ under the group homomorphism induced by the above quasi-nef sequence, we will see that $\widetilde N/N(G)$ can be embedded into $\bR^{\oplus n-s-2}$ as a discrete subgroup. This is a contradiction.

Now the nef and big class $A := L_1 + \dots + L_n$ satisfies the desired property.

\vskip 1pc \noindent
{\bf Acknowledgement.}
The authors would like to thank the referee for the suggestions which allow us to improve the presentation of the paper. The first and last authors are supported by ARFs of NUS.

\section{Fujiki and Lieberman type theorem} \label{section_Fu_Li}

Let $X$ be a compact K\"ahler manifold of dimension $n$.
Let $\omega$ be a K\"ahler form on $X$. Its class in $H^{1,1}(X,\bR) := H^{1,1}(X, \bC) \cap H^2(X, \bR)$ is denoted by $\{\omega\} $.
Let $\Aut_{\{\omega\} }(X)$ denote the group of all automorphisms $g$ of $X$ preserving $\{\omega\} $, i.e., $g^*\{\omega\} =\{\omega\} $.
Fujiki and Lieberman proved in \cite[Theorem 4.8]{Fujiki78} and \cite[Proposition 2.2]{Lieberman78} that
$\Aut_0(X)$ is a finite-index subgroup of $\Aut_{\{\omega\} }(X)$.

The purpose of this section is to prove a more general version of this theorem that we will use later.

Let $\phi$ be a differential $(p,p)$-form on $X$. We say that $\phi$ is {\it positive} if in any local holomorphic coordinates $z$ we can write $\phi(z)$ as a finite linear combination,
with non-negative functions as coefficients, of $(p,p)$-forms of type
$$(il_1(z)\wedge \overline{l_1(z)})\wedge \cdots \wedge (il_p(z)\wedge \overline{l_p(z)}),$$
where $l_j(z)$ are $\bC$-linear functions in $z$. This notion does not depend on the choice of local holomorphic coordinates $z$.

A $(p,p)$-current $T$ on $X$ is said to be {\it weakly positive} if $T\wedge\phi$ defines a positive measure for any positive $(n-p,n-p)$-form $\phi$.
A $(p,p)$-form is {\it weakly positive} if it is weakly positive in the sense of currents.
A $(p,p)$-current $T$ on $X$ is said to be {\it positive} if $T\wedge\phi$ defines a positive measure for any weakly positive $(n-p,n-p)$-form $\phi$.
It turns out that a form is positive if it is positive in the sense of currents and positivity implies weak positivity.
These two notions coincide for $p=0,1,n-1,n$ and are different otherwise.
If $\alpha$ is a holomorphic $p$-form on $X$, then $i^{p^2}\alpha\wedge\bar\alpha$ is an example of weakly positive $(p,p)$-form.
It vanishes only when $\alpha=0$. This can be easily checked using local coordinates on $X$.

Weakly positive and positive forms and currents are real, that is, they are invariant under the complex conjugation.
If $\omega$ is a K\"ahler form on $X$ as above, then $\omega^p$ is a positive $(p,p)$-form for every $p$.
A positive (resp. weakly positive) $(p,p)$-form or current $T$ is said to be {\it strictly positive} (resp. {\it strictly weakly positive})
if there is an $\epsilon>0$ such that $T-\epsilon\omega^p$ is positive (resp. weakly positive).
We refer to Demailly \cite{Demailly}, or \cite{DS10} for more details.

Recall from Hodge theory that the group $H^{p,q}(X, \bC)$ can be obtained as the quotient of the space of closed differential $(p,q)$-forms by the subspace of exact $(p,q)$-forms.
It is not difficult to see that every closed $(p,q)$-current $T$, defines via the pairing $\langle T,\phi\rangle$ with $\phi$ a closed $(n-p,n-q)$-form,
an element in the dual of $H^{n-p,n-q}(X, \bC)$, because $T$ vanishes on exact $(n-p,n-q)$-forms by Stokes' theorem.
So by Poincar\'e duality, $T$ defines a class in $H^{p,q}(X, \bC)$ that will be denoted by $\{T\}$.

Denote by $\CK_p(X)$ (resp. $\CK^w_p(X)$) the set of classes of strictly positive (resp. strictly weakly positive) closed $(p,p)$-forms.
They are strictly (i.e., salient) convex open cones in $H^{p,p}(X,\bR) := H^{p,p}(X,\bC) \cap H^{2p}(X,\bR)$.
Denote by $\CE_p(X)$ (resp. $\CE_p^w(X)$) the set of classes of positive (resp. weakly positive) closed $(p,p)$-currents.
They are strictly (i.e., salient) convex closed cones in $H^{p,p}(X,\bR)$. Indeed, if $T$ is a (weakly) positive closed $(p,p)$-current,
the quantity $\langle T,\omega^{n-p}\rangle$ is equivalent to the mass-norm of $T$ and it only depends on the cohomology class of $T$.
One often calls this quantity the {\it mass} of $T$. So $T$ vanishes if and only if its class $\{T\}$ vanishes.
Moreover, if $T_n$ are such that $\{T_n\}$ converge to some class $c$, then since $\{T_n\}$ are bounded, the masses of $T_n$ are bounded and hence,
up to extract a subsequence, we can assume $T_n$ converge to some (weakly) positive closed current $T$.
It follows that $c$ is represented by $T$ and hence $\CE_p(X)$ (resp. $\CE_p^w(X)$) are closed.

The {\it K\"ahler cone} of $X$, denoted by $\CK(X)$, consists of all K\"ahler classes.
So it is equal to $\CK_1(X)$ and also to $\CK_1^w(X)$ because positivity and weak positivity coincide for bidegree $(1,1)$.
The closure $\overline{\CK}(X) \subset H^{1,1}(X, \bR)$ of $\CK(X)$ is called the {\it nef cone}.
The cones $\CE_1(X)$ and $\CE_1^w(X)$ are also equal and are denoted simply by $\CE(X)$. This is called the {\it pseudo-effective cone}.
Classes in the interior of $\CE(X)$ are called {\it big} and are the classes of {\it K\"ahler currents}, i.e., strictly positive closed $(1,1)$-currents.
In general, we have
$$\CK_p(X)\subset \CK_p^w(X)\subset \overline \CK^w_p(X)\subset \CE^w_p(X)
\quad \text{and} \quad \CK_p(X)\subset \overline \CK_p(X)\subset \CE_p(X)\subset \CE_p^w(X).$$
The classes in the interior of $\CE_p(X)$ (resp. $\CE_p^w(X)$) are the classes of strictly positive (resp. strictly weakly positive) closed $(p,p)$-currents.
In what follows, we will write $c\leq c'$ or $c'\geq c$ for classes $c,c'\in H^{p,p}(X,\bR)$ such that $c'-c$ belongs to $\CE_p^w(X)$.

The automorphism group $\Aut(X)$ acts naturally on cohomology groups and preserves all the above cones, their closures, interiors and boundaries.
We will fix a norm $\|\cdot\|$ for each cohomology group and a K\"ahler form $\omega$ of $X$.
Our results do not depend on the choice of these norms or $\omega$. The main result in this section is the following theorem.

\begin{theorem} \label{th_Fu_Li}
Let $X$ be a compact K\"ahler manifold of dimension $n$ and $G \le \Aut(X)$ a group of automorphisms.
Assume that for every element $g\in G$ there is a class $c$ of a strictly weakly positive closed $(p,p)$-current
with $1\leq p\leq n-1$,
such that $\|(g^m)^*c\|=\oo(m)$ as $m\to +\infty$. Here, $c$ and $p$ may depend on $g$. Then $G$ is virtually contained in $\Aut_0(X)$, i.e., $|G : G \cap \Aut_0(X)| < \infty$.
\end{theorem}

The following immediate corollary can be applied when $c$ is a big class in $H^{1,1}(X,\bR)$. If $c$ is the class of a K\"ahler form,
the result is exactly the theorem by Fujiki and Lieberman quoted above.

\begin{corollary} \label{cor_Fu_Li}
Let $c$ be the class of a strictly weakly positive closed $(p,p)$-current, $1\leq p\leq n-1$.
Let $\Aut_c(X)$ denote the group of all automorphisms $g$ such that $g^*c=c$. Then $\Aut_c(X)$ is virtually contained in $\Aut_0(X)$.
\end{corollary}

We will see in Lemma \ref{lemma_Fu_Li} below that the condition that $g^*c=c$ for $c$ as in the above corollary is equivalent to the condition that $g^*c$ is parallel to $c$.

We prove now Theorem \ref{th_Fu_Li}.
We will identify $H^{n,n}(X,\bR)$ with $\bR$ by taking the integrals of $(n,n)$-forms on $X$. The cup product of two cohomology classes $c,c'$ is simply denoted by $cc'$ or $c\cdot c'$.
We need the following lemmas.

\begin{lemma} \label{lemma_Fu_Li}
Let $g\in\Aut(X)$. Assume there is a class $c$ of a strictly weakly positive closed $(p,p)$-current for some $1\leq p\leq n-1$, such that $\|(g^m)^*c\|=\oo(m)$ as $m\to +\infty$.
Then, for every $0\leq q\leq n$, the norm of $(g^m)^*$ acting on $H^{q,q}(X,\bR)$ is bounded by a constant independent of $m\in\bN$.
In particular, the conclusion holds if $g^*c$ is parallel to $c$ and in this case we have $g^*c=c$.
\end{lemma}

\begin{proof}
Using the Jordan canonical form of $g^*$ acting on $H^{p,p}(X,\bC)$,
we see that the hypothesis $\|(g^m)^*c\|=\oo(m)$ implies that $\|(g^m)^*c\|$ is bounded by a constant independent of $m\in\bN$.
Let $c_p$ be any class in $\CE_p^w(X)$. Since $c$ is in the interior of $\CE_p^w(X)$, there is a constant $\lambda>0$ such that $c_p\leq \lambda c$.
Since $g^*$ preserves $\CE_p^w(X)$ and this cone is salient and closed, we deduce that the sequence $\|(g^m)^*c_p\|$ is also bounded.
We will use this property for the class $\{\omega^p\}$ of $\omega^p$, where $\omega$ is the fixed K\"ahler form of $X$.

For $0\leq q\leq n$, we define
$$I_{q,m}:=\int_X(g^m)^*(\omega^q)\wedge\omega^{n-q} = (g^m)^*\{\omega^q\}\cdot \{\omega^{n-q}\} \quad \text{and} \quad
J_q:=\lim_{m\to +\infty} \frac{\log I_{q,m}}{\log m}\cdot$$
By the Jordan canonical form of $g^*$ acting on $H^{q,q}(X, \bC)$, we see that the limit in the definition of $J_q$ exists in $\bZ_{\ge 0}\cup\{\pm\infty\}$.
The above discussion implies that $J_p\leq 0$. We also have $J_0=J_n=0$ since $g^*$ is the identity on $H^{0,0}(X,\bR) \cong \bR$ and $H^{n,n}(X,\bR) \cong \bR$.

We apply Theorem \ref{th_Gromov} below to $\omega_1:=\omega$, $\omega_2:=(g^m)^*\omega$,
$\omega_i:=(g^m)^*\omega$ for arbitrary $q-1$ indices $3 \le i \le n$ and $\omega_i:=\omega$ for the other $n-q-1$ indices $3 \le i \le n$.
We obtain that $I_{q,m}^2\geq I_{q-1,m} I_{q+1,m}$. It follows that the function $q\mapsto J_q$ is concave, i.e., $2J_q\geq J_{q-1}+J_{q+1}$.
We then deduce that $J_q=0$ for every $0\leq q\leq n$. In particular, using again the Jordan canonical form of $g^*$ acting on $H^{q,q}(X, \bC)$, we see that the sequence $I_{q,m}$ is bounded.

Since the integral $I_{q,m}$ is the mass of the positive closed $(q,q)$-current $(g^m)^*(\omega^q)$, we deduce that the sequence $\|(g^m)^*\{\omega^q\}\|$ is also bounded.
In the same way as it was just done for $\CE_p^w(X)$ at the beginning of the proof, we obtain that the sequence $\|(g^m)^*c_q\|$ is bounded for any $c_q\in \CE_q^w(X)$.
Thus, the property holds for any $c_q\in H^{q,q}(X,\bR)$ because $\CE_q^w(X)$ spans $H^{q,q}(X,\bR)$. So the first assertion follows.

Suppose now that $g^*c = \lambda c$ for some $\lambda \in\bR$. Since $g^*$ preserves the weak positivity, we have $\lambda>0$.
Replacing $g$ by $g^{-1}$ if necessary, we may assume that $\lambda \le 1$.
Then the condition of the first assertion is satisfied, so the norm of $(g^m)^*$ acting on all $H^{q,q}(X, \bR)$ is bounded.
Thus the spectral radius of $g^*$ on all $H^{q,q}(X, \bR)$ are less than or equal to $1$.
The same is true for $(g^{-1})^*$, because the action of $(g^{-m})^*$ on $H^{q,q}(X, \bR)$ is the adjoint action of $(g^m)^*$ on $H^{n-q,n-q}(X,\bR)$ by Poincar\'e duality.
Thus, all eigenvalues of $g^*$ on all $H^{q,q}(X, \bR)$ are of modulus $1$. In particular, if $g^*c = \lambda c$, then $\lambda = 1$.
\end{proof}

Recall the following result of Gromov, that we used above, see \cite[Corollary 2.2]{Dinh12} or \cite{Gromov90}. It is a consequence of the mixed version of the classical Hodge--Riemann theorem.

\begin{theorem}[Gromov] \label{th_Gromov}
Let $\omega_i$ be K\"ahler forms on $X$, $1\leq i\leq n$. Define for $1\leq i,j\leq 2$
$$I_{ij}:=\int_X \omega_i\wedge \omega_j\wedge\omega_3\wedge \cdots\wedge\omega_n=\{\omega_i\}\{\omega_j\}\{\omega_3\}\cdots\{\omega_n\}.$$
Then we have $I_{12}^2\geq I_{11}I_{22}$.
\end{theorem}

\begin{lemma}\label{lemma_Fu_Li_bis}
Let $g \in \Aut(X)$.
Assume that the norm of $(g^m)^*$ acting on $H^{1,1}(X,\bR)$ is bounded by a constant independent of $m\in\bN$.
Then the action $g^*|_{H^2(X,\bR)}$ has finite order.
\end{lemma}

\begin{proof}
We first show that the action of $(g^m)^*$ on $H^2(X,\bR)$ is bounded independently of $m\in\bN$. Recall that $(g^m)^*$ preserves the Hodge decomposition
$$H^2(X,\bC)=H^{2,0}(X, \bC)\oplus H^{1,1}(X, \bC)\oplus H^{0,2}(X, \bC).$$
So it is enough to prove that the action of $(g^m)^*$ on $H^{2,0}(X, \bC)$ is bounded independently of $m\in\bN$.
The similar property for $H^{0,2}(X, \bC)$ is then obtained by complex conjugation.

Observe that if $v$ is a class in $H^{2,0}(X, \bC)$, it can be represented by a unique closed holomorphic 2-form $\alpha$.
The form $\alpha\wedge\bar\alpha$ is weakly positive and closed, which vanishes if and only if $\alpha=0$.
We deduce that $v \bar v=0$ if and only if $v=0$. Hence, $\|v\|$ is bounded if and only if $\|v \bar v\|$ is bounded.

For any $(2,0)$-class $v$ in $H^{2,0}(X, \bC)$, there is a K\"ahler class $u$ satisfying $v\bar v\leq u^2$.
Indeed, since the class $\{\omega^2\}$ is in the interior of the cone $\CE_2^w(X)$, it is enough to choose $u$ as a large multiple of the fixed K\"ahler class $\{\omega\}$.
We deduce that
$$(g^m)^*v\cdot (g^m)^* \bar v=(g^m)^*(v \bar v)\leq (g^m)^*(u^2)=(g^m)^*(u)^2.$$
Therefore, by hypothesis, $\|(g^m)^*v\cdot (g^m)^*\bar v \|$ is bounded independently of $m\in\bN$.
Thus $\|(g^m)^*v\|$ satisfies the same property.

We conclude that the action of $(g^m)^*$ on $H^2(X,\bR)$ is bounded independently of $m\in\bN$.
In particular, all eigenvalues of $g^* |_{H^2(X,\bR)}$ are of modulus $\le 1$, hence they are of modulus $1$, and indeed are roots of unity by Kronecker's theorem,
since $g^*$ preserves (and is an isomorphism of) $H^2(X,\bZ)$ and $H^2(X,\bR)=H^2(X,\bZ)\otimes_\bZ\bR$.
This and the above boundedness result imply that the action $g^* |_{H^2(X,\bR)}$ is periodic, by looking at the Jordan canonical form of the action.
\end{proof}

\begin{proof}[{\bf End of the proof of Theorem \ref{th_Fu_Li}.}]
By Lemmas \ref{lemma_Fu_Li} and \ref{lemma_Fu_Li_bis}, for any $g \in G$, the automorphism $g^* |_{H^2(X,\bR)}$ is periodic.
The characteristic polynomials of $g^* |_{H^2(X,\bR)}$ are defined over $\bZ$ and we have seen that their zeros are roots of unity for all $g \in G$.
Since these polynomials are of degree $\dim H^2(X,\bR)$, the orders of these roots of unity are bounded by the same constant, i.e.,
the group $G |_{H^2(X,\bR)}$ has bounded exponent. Thus, it is a finite group by the classical Burnside's theorem, see \cite{Burnside}.
Set $v := \sum h^*\{\omega\}$, where $h$ runs through the finite group $G |_{H^{1,1}(X,\bR)}$ and $\omega$ is the fixed K\"ahler form of $X$.
Then $G$ is contained in
$$\Aut_{v}(X) := \{g \in \Aut(X) \, | \, g^*v = v\}.$$
Since $v$ is a K\"ahler class, the theorem by Fujiki and Lieberman quoted above says that $|\Aut_{v}(X) : \Aut_0(X)| < \infty$.
Now
$$G/(G \cap \Aut_0(X)) \cong (G \cdot \Aut_0(X))/\Aut_0(X) \le \Aut_{v}(X)/\Aut_0(X)$$
and the latter is a finite group. The theorem follows.
\end{proof}

Since $\Aut_0(X)$ acts trivially on $H^{p,q}(X, \bC)$, the above discussion immediately implies the following proposition.

\begin{proposition}\label{LF}
Let $X$ be a compact K\"ahler manifold of dimension $n$ and $G \le \Aut(X)$ a group of automorphisms. Then $G$ is virtually contained in $\Aut_0(X)$ if and only if
the representation of $G$ on ${H^{p,p}(X,\bR)}$ for some $1\leq p\leq n-1$ (or on $\oplus_{0\leq p,q\leq n} H^{p,q}(X, \bC)$) is a finite group.
\end{proposition}

\section{Invariant privileged vector spaces} \label{section_privileged}

Let $X$ be a compact K\"ahler manifold of dimension $n$. By Hodge theory, every class $v$ in $H^2(X,\bR)$ admits a unique decomposition
$$v=v_{20}+v_{11}+v_{02},$$
where
$$v_{11}\in H^{1,1}(X,\bR), \,\, v_{20}\in H^{2,0}(X, \bC), \,\, v_{02}\in H^{0,2}(X, \bC)$$
with $v_{02}=\bar v_{20}$. For any vector subspace $V \subseteq H^2(X,\bR)$ define
$$V^0:=V\cap (H^{2,0}(X, \bC) \oplus H^{0,2}(X, \bC)), \,\, V^{11}:=V\cap H^{1,1}(X,\bR).$$
In general, we have $V^0\oplus V^{11}\subseteq V$, but the inclusion may be strict.

\begin{definition} \label{def_priv}
We say that $V$ is {\it privileged} if it satisfies the following properties:
\begin{itemize}
\item[(a)] \label{aaa} $V$ is defined over $\bQ$, $V\not=0$ and $V=V^0\oplus V^{11}$.
\item[(b)] \label{bbb} The intersection of $V^{11}$ with the nef cone $\overline{\CK}(X)$ has non-empty interior in $V^{11}$,
or equivalently this intersection spans $V^{11}$ (because $\overline{\CK}(X)$ is convex).
\item[(c)] \label{ccc} For every $v\in V^0$ there is a nef class $u$ in $V^{11}$ such that $v_{20}\bar v_{20}\leq u^2$.
\end{itemize}
\end{definition}

The main purpose of this section is to construct some privileged subspace of $H^2(X,\bR)$ as stated in the introduction. More precisely, we have the following result.

\begin{proposition}\label{priv}
Let $G$ and $Z$ be subgroups of $\Aut(X)$. Assume that $Z$ is normalized by $G$ and that $Z|_{H^2(X,\bC)}$ (or equivalently $Z|_{H^2(X,\bR)}$) is a unipotent commutative matrix group.
Then, there is a privileged subspace $F\subseteq H^2(X,\bR)$ which is $G$-stable such that the action of $Z$ on $F$ is the identity.
\end{proposition}

Recall here that the representation $Z|_{H^2(X,\bR)}$ of $Z$ is defined over $\bQ$ because $\Aut(X)$ acts on $H^2(X,\bZ)$ and $H^2(X,\bR) = H^2(X,\bZ) \otimes_{\bZ} \bR$.

\begin{notation} \label{not_pullback}
\rm From here till Proposition \ref{priv_bis}, for simplicity, the pull-back action of $g \in \Aut(X)$ on $H^*(X,\bC)$, will be denoted by the same $g$, instead of $g^*$.
So for $g,h\in \Aut(X)$ and $v\in H^*(X,\bC)$ the identity $(gh)^*v=h^*(g^*v)$ will be written as $(gh)v=h(g(v))$.
\end{notation}

\begin{lemma}\label{Lpriv}
There is a privileged $Z$-stable subspace $V\subseteq H^2(X,\bR)$ such that the action of $Z$ on $V$ is the identity.
\end{lemma}

\begin{proof}
Set $V_0 := H^2(X, \bR)$, and $\CK_0 := \overline{\CK}(X)$ which has non-empty interior in $V_0^{11} = H^{1,1}(X,\bR)$.
Take any $z\in Z$. Recall that the action of $z$ on $H^2(X,\bR)$ is unipotent, i.e., all eigenvalues are equal to 1.
Therefore, in what follows, for vector spaces, we can use bases defined over $\bQ$.

Let $l\geq 0$ be the minimal integer such that the norm of $z^m$ acting on $V_0$ satisfies $\|z^m\|=\OO(m^l)$ as $m\to +\infty$.
If $l=0$ then $z$ is the identity on $V_0$. This can be seen by using the Jordan canonical form of $z$. So we can take $V=V_0$ when the property $l=0$ holds for all $z$.

Otherwise, fix a $z \in Z$ such that $l\geq 1$. Define $$\pi:V_0\to V_0$$ to be the limit of $m^{-l}z^m$ and $V_1$ the image of $V_0$ by $\pi$.
Using the Jordan canonical form of $z$ with a basis of $V_0$ defined over $\bQ$, we see that $V_1\not=0$, $V_1\not=V_0$ and $V_1$ is defined over $\bQ$.
Since $z^m$ preserves the decomposition $V_0^0\oplus V_0^{11}$, our $V_1$ satisfies Property \hyperref[aaa]{\ref{def_priv} (a)}.

Moreover, $\pi(\CK_0)$ has non-empty interior in $V_1^{11}$ by open mapping theorem applied to $\pi : V_0^{11}\to V_1^{11}$.
By the definition of $\pi$, we have $\pi(\CK_0)\subseteq \CK_0$. Denote by $\CK_1$ the intersection $V_1^{11}\cap \CK_0$.
It has non-empty interior in $V_1^{11}$. So $V_1$ satisfies Property \hyperref[bbb]{\ref{def_priv} (b)}.

For any $v\in V_1^0$, there is a $v^\star\in V_0^0$ such that $v=\lim m^{-l} z^m(v^\star)$. Let $u^\star\in \CK_0$ such that $v^\star_{20}\bar v_{20}^\star\leq u^{\star 2}$.
Define $$u := \lim m^{-l} z^m(u^\star).$$ We have $v_{20}\bar v_{20}\leq u^2$ because $$u^2-v_{20}\bar v_{20}=\lim m^{-2l} z^m(u^{\star 2}-v^\star_{20}\bar v_{20}^\star).$$
Indeed, the class $$m^{-2l} z^m(u^{\star 2}-v^\star_{20}\bar v_{20}^\star)$$ is represented by some weakly positive closed $(2,2)$-current
and the cone $\CE^w_2(X)$ of weakly positive classes in $H^{2,2}(X, \bR)$ is closed. So Property \hyperref[ccc]{\ref{def_priv} (c)} holds for $V_1$. And eventually our $V_1$ is privileged.

For any arbitrary $w\in Z$, as actions on $H^2(X,\bR)$, we have $wz^m=z^mw$. It follows that $w\pi=\pi w$. So $V_1$ is $w$-stable for all $w \in Z$, and hence $Z$-stable.
We can now apply the same argument in order to construct inductively $V_{i+1}\subset V_i$ and $\CK_{i+1}\subset \CK_i$.
For dimension reason, there is an $i$ such that $Z$ is the identity on $V_i$. Take $V=V_i$. We are done.
\end{proof}

\begin{proof}[{\bf Proof of Proposition \ref{priv}.}]
Let $V$ be as in Lemma \ref{Lpriv}. We show that for any $g\in G$, $g(V)$ is also privileged.
Property \hyperref[aaa]{\ref{def_priv} (a)} is clear because $g$ is defined over $\bQ$ and preserves the Hodge decomposition $H^{2,0}(X,\bC)\oplus H^{1,1}(X,\bC)\oplus H^{0,2}(X,\bC)$.
Property \hyperref[bbb]{\ref{def_priv} (b)} is a consequence of the fact that $g$ is an isomorphism of $H^{1,1}(X,\bR)$ and preserves the nef cone $\overline{\CK}(X)$.
Property \hyperref[ccc]{\ref{def_priv} (c)} is also clear because $g$ preserves the weak positivity.
For any $z\in Z$ and $g \in G$, we have $z' := gzg^{-1}\in Z$ since $Z$ is normalized by $G$, and hence for any $v\in V$ (see Notation \ref{not_pullback})
$$z(g(v))=g(z'(v))=g(v).$$
Thus, the action of $Z$ on $g(V)$ is the identity.

Consider $F=\sum V$, where $V$ runs over all subspaces satisfying Lemma \ref{Lpriv}. We have $g(F)\subseteq F$ and hence $g(F)=F$ for dimension reason.
Also by dimension reason, the last sum is in fact a finite sum. Therefore, $F$ is defined over $\bQ$ because it has a system of generators defined over $\bQ$.
Like on $V$, the group $Z$ acts as the identity on $F$, and $F^{11}$ is spanned by nef classes.
Hence $F$ satisfies Properties \hyperref[aaa]{\ref{def_priv} (a)} and \hyperref[bbb]{\ref{def_priv} (b)}.

To finish the proof of Proposition \ref{priv}, we still have to show that $F$ satisfies Property \hyperref[ccc]{\ref{def_priv} (c)}.
For this purpose, it is enough to show that if $v_{20},v_{20}'\in H^{2,0}(X, \bC)$ and $u,u'\in \overline\CK(X)$ satisfy $v_{20}\bar v_{20}\leq u^2$ and $v_{20}'\bar v_{20}'\leq u'^2$,
then $(v_{20}+v_{20}')(\bar v_{20}+\bar v_{20}')\leq 2(u+u')^2$. We can represent $v_{20}-v'_{20}$ by a closed holomorphic 2-form $\alpha$.
Since the $(2,2)$-form $\alpha\wedge \bar\alpha$ is weakly positive, we deduce the following Cauchy--Schwarz inequality
$$(v_{20}-v'_{20})(\bar v_{20}-\bar v'_{20})\geq 0 .$$
It follows that
$$(v_{20}+v'_{20})(\bar v_{20}+\bar v'_{20})\leq 2v_{20}\bar v_{20}+2v'_{20}\bar v'_{20} \leq 2u^2+2u'^2.$$
On the other hand, we can approximate $u$ and $u'$ by classes of K\"ahler forms and hence $uu'$ by classes of positive $(2,2)$-forms.
Since the cone $\CE_2^w(X)$ is closed in $H^{2,2}(X,\bR)$, we obtain $uu'\geq 0$. Therefore, the right-hand side class in the latter chain of inequalities is bounded by $2(u+u')^2$.
This ends the proof of Proposition \ref{priv}.
\end{proof}

\begin{proposition}\label{priv_bis}
Let $G$, $Z$ and $F$ be as in Proposition \ref{priv}. Then, for every $g\in G$ there exists a nef class $v\ne 0$ in $F^{11} \subseteq F$
such that $gv=\lambda v$, where $\lambda$ is the spectral radius of $g$ acting on $F$.
\end{proposition}

\begin{proof}
If $\gamma$ is a complex eigenvalue of $g$ acting on $F^0$, then up to replace $\gamma$ with $\bar\gamma$,
there are vectors $v,v'\in F^0$ such that $v''_{20}:=v_{20}+iv'_{20}\not=0$ and $g(v''_{20})=\gamma v''_{20}$. Therefore, we have
$$g^m(v''_{20}\bar v''_{20})=|\gamma|^{2m} v''_{20}\bar v''_{20}.$$
Let $u,u'\in F^{11}$ be as in Property \hyperref[ccc]{\ref{def_priv} (c)} for $v,v'$ respectively. As in the proof of Proposition \ref{priv}, for $u''=\sqrt{2}(u+u')$, we have
$$v_{20}''\bar v_{20}''\leq u''^2 \quad \text{and hence} \quad |\gamma|^{2m} v''_{20}\bar v''_{20} \le g^m(u''^2).$$
It follows that
$$|\gamma|^m \lesssim \|g^m(u'')\|.$$ Therefore, $|\gamma| \le \rho(g|_{F^{11}})$, where $\rho(g|_{F^{11}})$ is the spectral radius of $g$ acting on $F^{11}$.
So $g$ acts on $F^{11}$ with spectral radius $\lambda$ and preserves the cone $C := F^{11}\cap\overline\CK(X)$.
Since this cone has non-empty interior in $F^{11}$, the proposition follows from Theorem \ref{B-cone} below.
\end{proof}

We quote Birkhoff's generalization of the Perron--Frobenius theorem.

\begin{theorem}[{cf. \cite{Birkhoff67}}] \label{B-cone}
Let $C$ be a strictly (i.e., salient) convex closed cone of a finite-dimensional $\bR$-vector space $V$ such that $C$ spans $V$ as a vector space.
Let $g : V \to V$ be an $\bR$-linear endomorphism such that $g(C) \subseteq C$.
Then the spectral radius $\rho(g)$ is an eigenvalue of $g$ and there is an eigenvector $L_g \in C$ corresponding to the eigenvalue $\rho(g)$.
\end{theorem}

The following lemma slightly generalizes \cite[Proposition 4.1]{DS04} which only deals with commutative groups.
We will use this lemma later to obtain the existence of common eigenvectors in the proof of Proposition \ref{PropA}.
\begin{lemma}\label{cB-cone}
Let $C$ be a strictly (i.e., salient) convex closed cone of a finite-dimensional $\bR$-vector space $V$ such that $C$ spans $V$ as a vector space.
Let $G \le \GL(V)$ and let $N \unlhd G$ be a finite normal subgroup such that $G(C) \subseteq C$ and the quotient group $G/N$ is commutative.
Then for every $g \in G$ there is a non-zero element $L_g$ in $C$ such that $g(L_g) = \rho(g) L_g$ with $\rho(g)$ the spectral radius of $g$,
$G(L_g) \subseteq \bR_{+} L_g$, and $n(L_g) = L_g$ for any $n \in N$.
\end{lemma}

\begin{proof}
Consider the linear map $\pi:V\to V$ defined by $$\pi(L):=\frac{1}{|N|} \sum_{n\in N} n(L),$$
and denote by $V'$ its image. This is the subspace of all vectors in $V$ which are fixed by $N$.
Hence, since $N$ is normal in $G$, it is not difficult to see that $V'$ is invariant by $G$.
Define $C':=C\cap V'=\pi(C)$. Since $C$ is closed, strictly convex and spans $V$, $C'$ is also closed, strictly convex, and spans $V'$.
By taking $L$ in the interior of $C$, we see that $C'$ contains vectors in the interior of $C$.

Observe that since $C$ is strictly convex and invariant by $G$, for every vector $L$ in the interior of $C$ and for any $g\in G$, when $n\to\infty$,
the growth of $\|g^n(L)\|$ is comparable with the growth of the norm $\|g^n\|$ of $g^n$ acting on $V$.
We then conclude that the norm $\|g^n\|$ of $g^n$ acting on $V'$ is comparable with the one on $V$ as $n\to\infty$.
In particular, the spectral radius of $g$ on $V'$ is equal to the one on $V$.

Now replacing $G,V,C$ by $G/N$, $V'$ and $C'$ reduces the problem to the case of the action of a commutative group. This case was proved in \cite[Proposition 4.1]{DS04}.
\end{proof}

Before the end of this section, we recall the cone theorem of Lie--Kolchin type in \cite{KOZ09}, which will be used in the proof of our main Theorem \ref{ThA}.

\begin{theorem}[{cf. \cite[Theorem 1.1]{KOZ09}}]
\label{Lie-Kolchin cone}
Let $V$ be a finite-dimensional real vector space and $\{0\} \ne C \subset V$ a strictly (i.e., salient) convex closed cone.
Suppose that $G \le \GL(V)$ is a solvable group, and it has connected Zariski closure in $\GL(V_\bC)$ (always possible by replacing $G$ with a finite-index subgroup) and $G(C) \subseteq C$.
Then $G$ has a common eigenvector in the cone $C$.
\end{theorem}

\begin{remark}
Note that in this theorem the connectedness of the Zariski closure of $G$ in $\GL(V_\bC)$ is necessary. Indeed, we have the following example.
Let $G$ be the subgroup of $\GL(2,\bR)$ consisting of all invertible $2\times2$ matrices $(a_{ij})_{1\leq i,j\leq 2}$ with non-negative entries
which are diagonal (i.e., $a_{12}=a_{21}=0$) or anti-diagonal (i.e., $a_{11}=a_{22}=0$).
Then $G$ is a solvable Lie group and its Zariski closure in $\GL(2,\bC)$ has two connected components.
It preserves the cone of vectors with non-negative coordinates. However, there is no common eigenvector of $G$ in this cone.
\end{remark}

\begin{remark}\label{rL-Kcone}
Let $V$ be a finite-dimensional real vector space and $\{0\} \ne C \subset V$ a strictly (i.e., salient) convex closed cone.
Suppose that $G \le \GL(V)\le \GL(V_\bC)$ is a virtually solvable group and $G(C) \subseteq C$.
Let $G_1 \unlhd G$ be the preimage, via the natural homomorphism $G \to \GL(V_{\bC})$,
of the identity connected component of the Zariski closure of $G$ in $\GL(V_{\bC})$.
This connected component is exactly the Zariski closure of $G_1$ in $\GL(V_{\bC})$ and is denoted by $\overline G_1$.
So $\overline G_1$ is solvable because it is virtually solvable and connected.
We then deduce that $G_1$ is solvable and we can apply Theorem \ref{Lie-Kolchin cone} to $G_1$.
\end{remark}

\section{Generalization of Theorem \ref{ThA'} and proof of Proposition \ref{PropA}} \label{section_main_th}

Theorem \ref{ThA'} will follow from the more general form below.
\begin{theorem}\label{ThA}
Let $X$ be a compact K\"ahler manifold of dimension $n \ge 2$ and $G \le \Aut(X)$ a group of automorphisms
such that the null-entropy subset $N(G)$ is a subgroup of (and hence normal in) $G$ and $G/N(G) \cong \bZ^{\oplus n-1}$.
Then we have:
\begin{itemize}
\item[(1)] The representation of $N(G)$ on $\oplus_{0\leq p,q\leq n} H^{p,q}(X, \bC)$ is a finite group. Hence $N(G)$ is virtually contained in $\Aut_0(X)$.
\item[(2)] Either $N(G)$ is a finite group and hence $G$ is virtually a free abelian group of rank $n-1$,
or $X$ is a complex torus and $G \cap \Aut_0(X) = N(G) \cap \Aut_0(X)$ is Zariski-dense in $\Aut_0(X)\cong X$.
\item[(3)] $G /(G \cap \Aut_0(X))$ and $G |_E$ are virtually free abelian groups of rank $n-1$ for any $G$-stable (real or complex)
vector subspace $E$ of $\oplus_{0\le p,q\le n} H^{p,q}(X, \bC)$ containing $H^{p,p}(X,\bR)$ for some $1\leq p\leq n-1$, e.g., $H^2(X,\bR)$ or $H^{1,1}(X,\bR)$.
\end{itemize}
\end{theorem}

\begin{remark}
The condition of Theorem \ref{ThA} (except the requirement for the dynamical rank of $G$) is, up to replace $G$ by a finite-index subgroup, equivalent to that
the representation $G|_{H^{1,1}(X, \bC)}$ (or $G |_{\NS_{\bR}(X)}$ when $X$ is projective) is virtually solvable (see Remark \ref{rThA'}, \cite[Theorem 1.2]{Zhang-Invent}, or \cite[Theorem 2.3]{CWZ14}).
\end{remark}

\begin{proof}[{\bf Proof of Theorem \ref{ThA} (1).}] \label{pfThA}
Since $N(G)$ is of null-entropy and preserves $H^2(X,\bZ)$,
the restriction $N(G)|_{H^2(X,\bC)}$ is virtually unipotent (cf. \cite[Theorem 2.2]{CWZ14}).
Indeed, $$U:=\{g\in N(G) : g^*|_{H^2(X,\bC)} \text{ is unipotent}\}$$
is a normal subgroup of $N(G)$ with finite-index.
Note that $U|_{H^2(X,\bC)}$ is automatically nilpotent (with nilpotency class $c$ say).
If it is trivial, the proof is easy (see the argument below).
Thus we assume it is non-trivial, and hence has a non-trivial centre $Z|_{H^2(X,\bC)}$ for some normal subgroup $Z$ of $U$.

Note that $Z\unlhd G$ (since $Z$ is a characteristic subgroup of $G$) and $Z|_{H^2(X,\bC)}$ is a unipotent commutative matrix group.
Then by Proposition \ref{priv}, there is a non-trivial $G$-stable privileged vector subspace $F_1 \subseteq H^2(X,\bR)$
(defined over $\bQ$) such that $Z|_{F_1}$ is trivial.
Hence as the image of the nilpotent group $U|_{H^2(X,\bC)}$ by the natural restriction, $U|_{F_1}$ is also a nilpotent group of nilpotency class at most $c-1$.
Repeating the above process several times, we may find a non-trivial $G$-stable privileged vector subspace $F\subseteq H^2(X,\bR)$ (defined over $\bQ$)
such that $U|_F$ is trivial, and hence $N(G)|_F$ is finite.
Set $F_\bZ := F \cap H^2(X,\bZ)$. Then $F = F_\bZ \otimes_\bZ \bR$.

Replacing $G$ by a finite-index subgroup, we may assume that $N(G)|_F$ is trivial, see \cite[Lemma 2.4]{CWZ14}.
In other words, if $\pi : G\to G|_F$
denotes the natural restriction, then $N(G)\unlhd \Ker\pi$.
Thus $G|_F\isom G/\Ker\pi$ is the quotient group of $G/N(G)\isom \bZ^{\oplus n-1}$ by $\Ker\pi/N(G)$, and hence is abelian.
As in \cite[Proof of Theorem I]{DS04}, let $s$ be the maximal number of non-trivial linearly independent characters $\chi_i:G|_F\to (\bR,+)$ corresponding to nonzero common nef classes $L_i\in F$ of $G|_F$, such that
$g^*L_i = \exp\chi_i(g^*|_F)L_i$ for all $g \in G$.
By taking the composition with $\pi:G\to G|_F$, these $\chi_i$ can be thought of as characters of $G$ itself.
Define a group homomorphism as follows
$$\psi:G|_F\to (\bR^{\oplus s},+), \quad g^*|_F\mapsto (\chi_1(g),\dots,\chi_{s}(g)).$$
One can show that $\Ker\psi = N(G|_F)$ (cf. \cite[Proposition 4.2]{DS04}),
and $\Imm \psi$ is a discrete subgroup of $\bR^{\oplus s}$
(cf. \cite[Corollary 2.2]{DS04}) because our $F_{\bZ}$ is defined over $\bQ$.
Further, the lattice $\Imm \psi$ generates $\bR^{\oplus s}$. In other words, we have
$$G|_F/N(G|_F)\isom \bZ^{\oplus s}.$$
If $s = 0$, let $L_{s+1}$ be any nonzero common nef eigenvector of the commutative group $G|_F$.
If $s > 0$,
since $\psi(G)$ is a spanning lattice,
one may choose an automorphism $f^*|_F\in G|_F$ such that every coordinate of $\psi(f^*|_F)$ is negative;
let $L_{s+1} \in F$ be a nonzero common nef eigenvector of $G|_F$ such that $f^*L_{s+1}=\rho(f^*|_F)L_{s+1}$, where $\rho(f^*|_F)$ denotes the spectral radius of $f^*|_F$ (cf. \cite[Proposition 4.1]{DS04}, Proposition \ref{priv_bis}).
Note that $\rho((f^{-1})^*|_F) > 1$ by the choice of $f$, and hence $\rho(f^*|_F) > 1$, since $F_\bZ$ is defined over $\bZ$ and $G$-stable entailing $\det (f^*|_F) = \pm 1$.

As in \cite{DS04}, we claim that $L_1 \cdots  L_{s+1}\ne 0$.
Indeed, we first prove that $L_1 \cdots  L_{s}\ne 0$.
Suppose to the contrary that $L_1 \cdots  L_{s} = 0$.
Let $2\le k\le s$ be the least number such that $L_{i_1} \cdots L_{i_k} = 0$ for some $\{i_1,\dots,i_k\}\subseteq \{1,\dots,s\}$.
Then $L_{i_1} \cdots L_{i_{k-1}} \ne 0$, $L_{i_1} \cdots L_{i_{k-2}} \cdot L_{i_k} \ne 0$, and $L_{i_1} \cdots L_{i_k} = 0$,
imply that $\chi_{i_{k-1}}(g^*|_F) = \chi_{i_k}(g^*|_F)$ for any $g^*|_F\in G|_F$ by \cite[Lemma 4.4]{DS04}.
This is a contradiction since $\Imm \psi$ generates $\bR^{\oplus s}$.

We return back to the proof of the claim.
Suppose to the contrary that $L_1 \cdots L_{s+1} = 0$.
Since $L_1 \cdots L_{s}\ne 0$ and $\exp\chi_{s}(f^*|_F)<1<\rho(f^*|_F)$,
we have $L_1 \cdots L_{s-1} \cdot L_{s+1} = 0$ by \cite[Lemma 4.4]{DS04}.
Repeating this process, we would reach that $L_1 \cdot L_{s+1}=0$, hence $L_1$ and $L_{s+1}$ are parallel to each other (cf. \cite[Corollary 3.2]{DS04}).
But this is impossible by looking at the action of $f^*|_F$ on them.

We next claim that the above $s=n-1$.
Suppose to the contrary that $s<n-1$.
Extend $L_1 \cdots L_{s+1}$ to a full quasi-nef sequence
$$L_1 \cdots L_k \, \in \, \overline{L_1 \cdots L_{k-1} \cdot \Nef(X)} \, \subseteq \, H^{k,k}(X, \bR)$$
($1 \le k \le n$)
as in \cite[\S 2.7]{Zhang-Invent}, such that for any $g\in G$,
$$g^*(L_1 \cdots  L_k) = \exp\chi_1(g)  \cdots  \exp\chi_k(g) L_1 \cdots  L_k,$$
where $\chi_i:G\to (\bR,+)$ are group characters.
Then the group homomorphism
$$\varphi : G \to (\bR^{\oplus n-1}, +), \quad g \mapsto (\chi_1(g), \dots, \chi_{n-1}(g))$$
has $\Ker\varphi = N(G)$ and image discrete in $\bR^{\oplus n-1}$.

Denote by $\widetilde N$ the inverse image $\pi^{-1}\big(N(G|_F)\big)$ of $N(G|_F)$
under the natural restriction $\pi:G\to G|_F$.
Then $\widetilde N/N(G)$ is a free abelian subgroup of $G/N(G)\isom \bZ^{\oplus n-1}$ of rank $n-1-s$, because $$G/\widetilde N\isom \pi(G)/\pi(\widetilde N) = G|_F/N(G|_F)\isom \bZ^{\oplus s}.$$
On the other hand, since the null-entropy group $N(G|_F)$ of course fixes those $s+1$ common nef classes $L_1,\dots,L_{s+1}\in F$ of $G|_F$,
all the first $s+1$ coordinates of $\varphi(\widetilde N)$ are zero.
Hence one can embed $\varphi(\widetilde N)\isom \widetilde N/N(G)$ into $\bR^{\oplus n-1-(s+1)}$.
Recall that $\Imm \varphi$ is discrete in $\bR^{\oplus n-1}$.
Therefore, $\varphi(\widetilde N)$ is discrete in $\bR^{\oplus n-s-2}$, and hence of rank $\le n-s-2$,
which contradicts the rank calculation above.
So the claim holds.

Therefore, we obtain $n$ common nef classes $L_i$ of $G$, which are of course fixed by the null-entropy group $N(G)$.
Thus $N(G)$ fixes the nef class $A:=L_1+ \cdots +L_n$, 
which is also big because $L_1 \cdots L_n$ is nonzero and hence positive.
Hence $N(G)$ is virtually contained in $\Aut_0(X)$ by Corollary \ref{cor_Fu_Li}.
Then Theorem~\ref{ThA} (1) follows from Proposition \ref{LF}.
\end{proof}

\begin{proof}[{\bf Proof of Theorem \ref{ThA} (2).}]
By the assertion (1), $N(G)$ is virtually contained in $\Aut_0(X)$, i.e., $|N(G) : N(G) \cap \Aut_0(X)| < \infty$.
Since $G/N(G) \cong \bZ^{\oplus n-1}$ and by \cite[Lemma 2.4]{CWZ14}, we only have to consider the case that the group $N(G) \cap \Aut_0(X)$ is infinite.
We will prove that $X$ is a complex torus in this case.

To do so, let $H$ be the identity connected component of the Zariski closure of the latter group in $\Aut_0(X)$. Then $H$ is normalized by $G$ because $N(G)\cap\Aut_0(X)$ is normal in $G$.
Consider the quotient map $X \dashrightarrow X/H$ (and its graph) as in \cite[Lemma 4.2]{Fujiki78}, where $G$ acts on $X/H$ bi-regularly.
The maximality of the dynamical rank $r(G)$ of $G$ implies that $H$ has a Zariski-dense open orbit in $X$, see \cite[Lemma 2.14]{Zhang-Invent}.

Consider the case that the irregularity $q(X) = 0$. Then $\Aut_0(X)$ is a linear algebraic group, see \cite[Theorem 5.5]{Fujiki78} or \cite[Theorem 3.12]{Lieberman78}.
Hence $X$ is an almost homogeneous variety under the action of a linear algebraic group $H\ (\le \Aut_0(X))$.
Then $|\Aut(X) : \Aut_0(X)| < \infty$ by \cite[Theorem 1.2]{FZ13}. Hence $G \le \Aut(X)$ is of null entropy, contradicting the assumption $r(G) = n-1 \ge 1$ on the dynamical rank of $G$.

So we may assume that $q(X) > 0$. Then the Albanese map $\alb_X : X \to A := \Alb(X)$ is bimeromorphic by the maximality of $r(G)$, see \cite[Lemma 2.13]{Zhang-Invent}.
Since the action $H |_A$ of $H$ on $A$ also has a Zariski-dense open orbit in $A$, we have $H |_A = \Aut_0(X)\cong A$.
Let $B \subset A$ be the locus over which $\alb_X$ is not an isomorphism. Note that $B$ and its inverse in $X$ are both $H$-stable.
Since $H |_A = \Aut_0(X)\cong A$, we have $B = \varnothing$. Hence $\alb_X$ is an isomorphism. This proves Theorem \ref{ThA} (2).
\end{proof}

\begin{proof}[{\bf Proof of Theorem \ref{ThA} (3).}]
Let $\phi$ be one of the following natural homomorphisms:
$$G \to G/(G \cap \Aut_0(X)) \quad \text{or} \quad G \to G |_E.$$
Then $K := \Ker(\phi) \le \Aut_{c}(X)$ for any class $c$ in $H^{p,p}(X,\bR)$. Choose a class $c$ in the interior of $\CE^w_p(X)$.
Since ${|}\Aut_{c}(X) {:} \Aut_0(X){|} < \infty$, by Corollary \ref{cor_Fu_Li}, this $K$ is of null entropy, see Proposition \ref{LF}. Thus $K \le N(G)$.
Now
$$\bZ^{\oplus n-1} \cong G/N(G) \cong (G/K)/(N(G)/K) \cong \phi(G)/\phi(N(G)).$$
By Theorem \ref{ThA} (1), $\phi(N(G))$ is a finite group. So there is a finite-index subgroup $G_1$ of $G$ such that $\phi(G_1) \cong \bZ^{\oplus n-1}$, see \cite[Lemma 2.4]{CWZ14}.
This proves Theorem \ref{ThA} (3), hence the whole of Theorem \ref{ThA}.
\end{proof}

Now the finiteness of $N(G)$ in \cite[Theorem I]{DS04} when $r(G) = n-1$, has a shorter proof:

\begin{corollary}\label{CorC}
Let $X$ be a compact K\"ahler manifold of dimension $n \ge 2$ and $G$ a group of automorphisms of $X$. Assume that $G$ is commutative of maximal dynamical rank $n-1$.
Then $N(G)$ is a finite group and hence $G$ is virtually isomorphic to $\bZ^{\oplus n-1}$.
\end{corollary}

\begin{proof}
Suppose the contrary that $N(G)$ is infinite. By Theorem \ref{ThA} (2), $X$ is a complex torus and $N(G) \cap \Aut_0(X)$ is Zariski-dense in $\Aut_0(X)$.
Take $g_0 \in G$ of positive entropy. Since $g_0$ commutes with all elements in $N(G)$, it commutes with any translation $T_b \in \Aut_0(X)$.
Write $g_0 = h \circ T_a$ for some translation $T_a$ and group automorphism $h$ for $X$ (here consider $X$ as an additive group). Then for every $x \in X$, we have
$$h(a + b + x) = (g_0 \circ T_b)(x) = (T_b \circ g_0)(x) = b + h(a + x).$$
Hence $h(b) = b$ for all $b \in X$. Thus $h = \id_X$. So $g_0 = T_a$, and $g_0$ is of null entropy, a contradiction. This proves Corollary \ref{CorC}.
\end{proof}

\begin{proof}[{\bf Proof of Proposition \ref{PropA}.}]
In Theorem \ref{Z-TitsTh}, we can take $G_1$ as the preimage in $G$ of the identity connected component of the Zariski closure of $G|_{H^{1,1}(X, \bC)}$ in $\GL(H^{1,1}(X, \bC))$,
via the natural homomorphism $G \to G|_{H^{1,1}(X, \bC)}$, see \cite[Theorem 1.2]{Zhang-Invent} and Remark \ref{rL-Kcone}.
Theorem \ref{ThA} can be applied to $G_1$ instead of $G$.
So $N(G_1)|_{H^{1,1}(X, \bR)}$ is a finite group and
$$G_1/N(G_1) \cong ({G_1}|_{H^{1,1}(X, \bR)})/(N({G_1})|_{H^{1,1}(X, \bR)}) \cong \bZ^{\oplus n-1} .$$
Further, $G_1 \unlhd G$ and $|G : G_1| < \infty$. 
For simplicity, we still use the same notation introduced in the proof of Theorem \ref{ThA}.
Applying Lemma \ref{cB-cone} to ${G_1}|_{H^{1,1}(X, \bR)}$ and the nef cone of $X$, we may assume that nef classes $L_i$ with $1\le i\le n$ are common eigenvectors for the action of $G_1$ on $H^{1,1}(X, \bR)$.
These $n$ classes $L_i$ form a basis of an $n$-dimensional real vector space that we denote by $W\subseteq H^{1,1}(X,\bR)$.

For any fixed index $1\le i\le n$, we can find $g_i \in G_1$
such that $g_i^*L_j = \lambda_{i, j} L_j$ with $\lambda_{i, j} < 1$, for any $j \ne i$.
This is true for $i = n$ because $\varphi(G_1)$ is a spanning lattice in $(\bR^{\oplus n-1}, +)$.
The same is true for other $i$ since $\sum_{j=1}^n \chi_j = 0$ or $\prod_{j=1}^n \exp \chi_j = 1$. 
Since $W$ contains a nef and big class $A:=\sum_{j=1}^n L_j$, 
we have $d_1(g_i) = \rho(g_i^*|_W)$, and hence $d_1(g_i) = \max_{j} \{\lambda_{i,j}\}$.
Since $\lambda_{i, j} < 1$ for $j \ne i$ and $\lambda_{i, i} = 1/\prod_{j \ne i} \lambda_{i, j} > 1$, we have $\lambda_{i,i}=d_1(g_i)$.
Further, we have the following uniqueness property by \cite[Lemma 4.4 and Corollary 3.2]{DS04}.

\medskip
\noindent
{\bf Uniqueness Property.}
Let $N$ be a nef class such that $g_i^*N = \lambda N$ with $\lambda >1$. Then $N$ is parallel to $L_i$.

\begin{claim} \label{claim_GG1}
For every $g\in G$, $g^*$ permutes the half-lines $\bR_+L_i$ for $i=1, \ldots, n$. In particular, $W$ is invariant by $G$.
\end{claim}

\begin{proof} \renewcommand{\qedsymbol}{}
Since $G_1$ is normal in $G$, we have $h_i:=g^{-1}g_ig\in G_1$. It follows that all $L_j$ are eigenvectors of $h_i^*$.
In particular, $W$ is $h_i^*$-stable, and if $\lambda$ is the spectral radius of $h_i^*|_{W}$, then $h_i^*L_j=\lambda L_j$ for some $j$.
Note that the eigenvalue associated with $L_j$ is always a positive number because $L_j$ is nef.
Since $W$ contains big and nef classes, we deduce that $\lambda$ is the first dynamical degree of $h_i$.
Moreover, since $h_i$ and $g_i$ are conjugated, they have the same first dynamical degree, i.e., $\lambda=d_1(g_i)$.

We deduce from the definition of $h_i$ that $(g^{-1})^*L_j$ is an eigenvector of $g_i^*$ associated with the maximal eigenvalue $\lambda=d_1(g_i)$.
By the above uniqueness property, $(g^{-1})^*L_j$ is parallel to $L_i$. It follows that $g^*L_i$ is parallel to $L_j$. This proves the claim.
\end{proof}

We go back to the proof of Proposition~\ref{PropA}. By Claim \ref{claim_GG1}, we can associate $g$ with an element $\sigma_g$ in the symmetric group $S_n$ that we identify with the group of all permutations of the half-lines $\bR_+L_i$.
So we have a natural group homomorphism from $G$ to $S_n$. Let $G_0$ denote its kernel. Recall that we use the basis $L_1,\ldots,L_n$ for the vector space $W$. We see that
$G_0$ is the set of $g\in G$ such that the action of $g^*$ on $W$ is given by an $n\times n$ diagonal matrix with non-negative entries and determinant $1$.
Indeed, the entries of this matrix are non-negative because the classes $L_j$ are nef;
the determinant is $1$ because $g^*$ is identity on $H^{n,n}(X,\bR)$ and hence preserves the class $L_1 \cdots  L_n$.
We also see that $N(G_0)$ is the set of all elements $g\in G_0$ which act as the identity on $W$, see Theorem \ref{th_Fu_Li}. Hence $N(G_0)$ is a normal subgroup of $G_0$.

Now, $G_0/N(G_0)$ can be identified with a subgroup of the group of all the diagonal $n\times n$ matrices with non-negative entries and determinant $1$.
The later one is isomorphic to the torsion-free additive group $(\bR^{\oplus n-1},+)$.
Moreover, since $G_0$ contains the finite-index subgroup $G_1$, the group $G_0/N(G_0)$ contains a finite-index subgroup isomorphic to $G_1/N(G_1)\cong\bZ^{\oplus n-1}$.
We necessarily have $G_0/N(G_0)\cong\bZ^{\oplus n-1}$. This completes the proof of Proposition \ref{PropA}.
\end{proof}

\begin{proof}[{\bf Proof of Remark \ref{rPropA}.}]
As in the proof of Proposition \ref{PropA}, let $G_1$ be the preimage in $G$ of the identity connected component of the Zariski closure of
$G|_{H^{1,1}(X, \bC)}$ in $\GL(H^{1,1}(X, \bC))$, via the natural homomorphism $G \to G|_{H^{1,1}(X, \bC)}$.
Then we have $G_1/N(G_1) \cong \bZ^{\oplus n-1}$ and $|G : G_1| < \infty$.
Further, by definition, $G_1$ is normalized by $H$. Now we follow the proof of Proposition \ref{PropA}, with $G$ there, replaced by $H$ here.
In particular, Claim \ref{claim_GG1} holds for all $g \in H$.
In the same way as for $G_0$, we construct a group $H_0$ such that $G_1 \unlhd H_0 \unlhd H$, $H/H_0$ is identified with a subgroup of $S_n$
and $H_0/N(H_0)$ is isomorphic to a subgroup of the torsion-free additive group $(\bR^{\oplus n-1}, +)$.
Hence by \cite[Lemma 5.5]{Dinh12}, ${H_0}|_{H^{1,1}(X, \bR)}$ is virtually solvable since ${N(H_0)}|_{H^{1,1}(X, \bR)}$ is virtually unipotent (cf. \cite[Theorem 2.2]{CWZ14}).
So, according to \cite[Theorem 1.2]{Zhang-Invent}, the torsion-free abelian group $H_0/N(H_0)$ is isomorphic to $\bZ^{\oplus n-1}$ since it contains $G_1/N(G_1) \cong \bZ^{\oplus n-1}$.
By Theorem \ref{ThA}, $N(H_0)$ and $N(G_1)$ are virtually contained in $\Aut_0(X)$.
Thus we have a commutative diagram:
\[\xymatrix{
H/(H \cap \Aut_0(X)) & & H_0/(H_0 \cap \Aut_0(X)) \ar@{_{(}->}[ll]_{\textrm{finite-index}} \ar@{>>}[rr]^{\textrm{finite-kernel}} & & H_0/N(H_0) \cong \bZ^{\oplus n-1} \\
G/(G \cap \Aut_0(X)) \ar@{^{(}->}[u]_{} & & G_1/(G_1 \cap \Aut_0(X)) \ar@{^{(}->}[u]_{} \ar@{_{(}->}[ll]_{\textrm{finite-index}} \ar@{>>}[rr]^{\textrm{finite-kernel}} & & G_1/N(G_1) \cong \bZ^{\oplus n-1} \ar@{^{(}->}[u]^{\textrm{finite-index}}.
}\]
It follows that $H/(H \cap \Aut_0(X))$ is a finite extension of $G/(G \cap \Aut_0(X))$ and both are virtually free abelian groups of rank $n-1$, see \cite[Lemma 2.4]{CWZ14}.
If $G_1 \cap \Aut_0(X)$ (or equivalently $N(G_1)$) is infinite, then it is Zariski-dense in $\Aut_0(X)$ and $X$ is a complex torus, see Theorem \ref{ThA}.
The same property holds if we replace $G_1$ by $H_0$. Otherwise, $|H_0 : G_1| < \infty$. The remark follows.
\end{proof}

\end{document}